\documentclass{amsart}
\usepackage{amsfonts}
\usepackage{amsmath}
\usepackage{amscd}
\usepackage{amssymb}
\usepackage{amsxtra}
\usepackage{latexsym}
\usepackage{amsthm}
\usepackage{mathabx}
\usepackage{graphicx}
\usepackage[bookmarks,colorlinks,pagebackref]{hyperref} 
\input{xy}
\xyoption{all}

\theoremstyle{plain}

\newtheorem*{conjectuur*}{Conjecture}

\swapnumbers
\newtheorem{theorem}[subsection]{Theorem}
\newcommand\Thm[1]{Theorem~\ref{#1}}
\newtheorem{corollary}[subsection]{Corollary}
\newcommand\Cor[1]{Corollary~\ref{#1}}
\newtheorem{lemma}[subsection]{Lemma}
\newcommand\Lem[1]{Lemma~\ref{#1}}
\newtheorem{proposition}[subsection]{Proposition}
\newcommand\Prop[1]{Proposition~\ref{#1}}

\newtheorem{conjecture}[subsection]{Conjecture}

\newtheorem*{maintheorem}{Main Theorem}

\theoremstyle{definition}
\newtheorem{definition}[subsection]{Definition}

\newtheorem{example}[subsection]{Example}

\theoremstyle{remark}

\newtheorem{remark}[subsection]{Remark}

\swapnumbers

\newcommand{\emptyprop}{q}

\newcommand \ann[2]{\operatorname{Ann}_{#1}(#2)}

\newcommand \ch{characteristic}
\newcommand \CM{Coh\-en-Mac\-au\-lay}

\newcommand \ext[4]{\operatorname{Ext}_{#1}^{#2}(#3,#4)}
\renewcommand \hom [3]{\operatorname{Hom}_{#1}(#2,#3)} 
\newcommand \homo{homomorphism}

\renewcommand\iff{if and only if}

\newcommand \inv[1]{{#1^{-1}}}

\newcommand \iso{\cong}

\newcommand \map[1]{{\newcommand{\tmpprop}{#1q}  \if\tmpprop\emptyprop \to\else \xrightarrow{{\phantom{i}{#1}\phantom{i}}}\fi}} 
\newcommand \maxim{\mathfrak m}
\newcommand \nat{\mathbb N}

\newcommand \op\operatorname
\newcommand \pol[2]{#1[#2]}
\newcommand \pow[2]{#1[[#2]]}

\newcommand \range [2]{#1,\dots,#2}

\newcommand \rij[2]{(#1_1,\dots,#1_{#2})}
\let\sub\subseteq

\newcommand \tensor{\otimes}

\newcommand \exactseq [5]{0\to{#1}\:\map{#2}\:{#3}\:\map{#4}\:{#5}\to0}
\newcommand \Exactseq [3]{0\to {#1}\to {#2}\to {#3}\to 0}

\newcommand{\commdiagram}[9][]{%
\begin{equation}
{\newcommand{\tmpprop}{#1q} 
\if\tmpprop\emptyprop \relax\else \label{#1}\fi}
\begin{aligned}%
\mbox{
\begin{picture}(130,90)%
\put(120,70){\vector( 0,-1){50}}%
\put(10,80){\vector( 1, 0){100}}%
\put(0,70){\vector( 0,-1){50}}%
\put(10,10){\vector( 1, 0){100}}%
\put(115,80){\makebox(0,0)[l]{$#4$}}%
\put(5,80){\makebox(0,0)[r]{$#2$}}%
\put(115,10){\makebox(0,0)[l]{$#9$}}%
\put(5,10){\makebox(0,0)[r]{$#7$}}%
\put(-3,50){\makebox(0,0)[r]{$#5$}}
\put(123,50){\makebox(0,0)[l]{$#6$}}
\put(60,3){\makebox(0,0)[c]{$#8$}}
\put(60,88){\makebox(0,0)[c]{$#3$}}
\end{picture}}
\end{aligned}
\end{equation}}

\newcommand\commtrianglefront[7][]{%
\begin{equation}
{\newcommand{\tmpprop}{#1q} 
\if\tmpprop\emptyprop \relax\else \label{#1}\fi}
\begin{aligned}%
\mbox{
\begin{picture}(120,80)%
\put(55,68){\vector(-1,-2){30}}
\put(65,68){\vector(1,-2){30}}
\put(30,5){\vector(1,0){60}}
\put(60,75){\makebox(0,0)[c]{$#2$}}
\put(25,5){\makebox(0,0)[r]{$#4$}}
\put(95,5){\makebox(0,0)[l]{$#6$}}
\put(60,0){\makebox(0,0)[c]{$#5$}}
\put(37,43){\makebox(0,0)[r]{$#3$}}
\put(83,43){\makebox(0,0)[l]{$#7$}}
\end{picture}}
\end{aligned}
\end{equation}}

\newcommand\commtriangleback[7][]{%
\begin{equation}
{\newcommand{\tmpprop}{#1q}
\if\tmpprop\emptyprop \relax\else \label{#1}\fi}
\begin{aligned}%
\mbox{
\begin{picture}(120,80)%
\put(55,70){\vector(-1,-2){30}}
\put(65,70){\vector(1,-2){30}}
\put(30,5){\vector(1,0){60}}
\put(60,75){\makebox(0,0)[c]{$#2$}}
\put(25,5){\makebox(0,0)[r]{$#6$}}
\put(95,5){\makebox(0,0)[l]{$#4$}}
\put(60,0){\makebox(0,0)[c]{$#5$}}
\put(37,43){\makebox(0,0)[r]{$#7$}}
\put(83,43){\makebox(0,0)[l]{$#3$}}
\end{picture}}
\end{aligned}
\end{equation}}

 \newcommand\fr{Frobenius transform}
 \newcommand\persym{persymmetric}
  \newcommand\pcan{pseudo-canonical}
\newcommand \trans[1]{\mat{#1}^{\text T}}
\newcommand \scal[2]{#1_{#2}^\wedge}

\newcommand\qdig[2]{\mathring{#1}_{#2}}
 \newcommand\mulel[2]{{#1}_{#2}^\times}
  \newcommand\quot[1]{\texttt q(#1)}
    \newcommand\rem[1]{\texttt r(#1)}
\newcommand\can[2]{\canring{#1}{#2}{}}
\newcommand\canring[3]{\mathbf K_{#2}^{#3}(#1)}

\newcommand\unmix[1]{{#1}^{\text{unm}}}

\newcommand\mat[1]{\mathbb{#1}}

\renewcommand\dim[1]{\op{dim}(#1)}

\newcommand\frob[1]{\mathbf{F}_{#1}}
\newcommand\ndo[1]{\op{End}(#1)}

\newcommand\fl[1]{\mathfrak{#1}}
\newcommand\matlis[1]{#1^\vee}
\newcommand\frobmatlis[1]{\mathbf{T}(#1)}

\newcommand\matlisring[2]{#2^\vee_{#1}}
\newcommand\class[1]{[#1]}

\newcommand\hlc[3]{\op{H}_{#3}^{#2}(#1)}

\newcommand\mul[2]{{\ell}_{#2}(#1)}
\newcommand\tuple[1]{\mathbf{#1}}

\newcommand \len[1]{\ell(#1)}

\title {Hochster's small MCM conjecture for three-dimensional weakly F-split rings}
\author{Hans Schoutens}
\date\today
\address{Department of Mathematics\\
NYC College of Technology and
the CUNY Graduate Center\\
New York, NY, USA}
\email{hschoutens@citytech.cuny.edu}
\subjclass[2010]{13D22,13D45,13A35}
\begin{document}

\begin{abstract} 
We prove Hochster's small MCM conjecture for three-dimensional complete F-pure rings.
We deduce this from  a more general criterion, and show that  only a weakening of the   notion of F-purity is needed, to wit, being weakly F-split.  We conjecture that any complete ring is weakly F-split.
  \end{abstract}
\maketitle



\section{Introduction}
Let $R$ be a $d$-dimensional  Noetherian local ring with residue field $k$. A module $M$ is called a \emph{small MCM} (=finitely generated maximal \CM\ module), if it has depth $d$. Since all modules will tacitly be assumed to be finitely generated, we drop the modifier `small' altogether. Without any further assumption on $R$, these might not exist, but Hochster conjectured that any complete local ring admits an MCM. The first unknown case of Hochster's conjecture is in dimension three.     We  proved in \cite{SchToric} that if   MCM's exist such that their multiplicities are not too big, then the equal \ch\ zero case follows from the positive \ch\ case; in the mixed \ch\ case, we even know less.  In this paper, we will mainly tackle the case  that $R$ is complete, of positive \ch\ $p$, and has dimension  $d=3$.  I now briefly describe the strategy to obtain a small MCM in this setting.

\subsection{Higher \pcan\ modules} 
If $(R,\maxim)$ is complete and \CM, it admits a canonical module $\omega$. Without the \CM\ assumption, we lack   Grothendieck vanishing, forcing us to define a sequence of modules instead: given a module $M$, we define its \emph{$i$-th higher \pcan\ module} $\can Mi$ as the Matlis dual of the local cohomology module $\hlc M{d-i}\maxim$. Of course, if $M$ is MCM, then all higher \pcan\ modules $\can Mi$, for $i=1,\dots,d$, vanish. Our fist main result (without any assumption on the \ch) is a substantial   weakening of this:

\begin{theorem}\label{T:pcanMCM}
A three-dimensional complete local ring admits an MCM \iff\ there exists a module $M$ such that $\can M1$ has positive depth.
\end{theorem} 
In fact, if $M$ satisfies the above assumption, then $\can M0$ is an MCM.  We turn now to \ch\ $p>0$ to construct such modules $M$.

\subsection{The \fr\ of a module}
Let $\frob q$ be the Frobenius morphism $x\mapsto x^q$, where $q=p^n$ is some power of the \ch\  (when $q$ is clear, we drop reference to it). Given an $R$-module $M$, we denote its pull-back via $\frob q$ by $\frob{q*}M$ or just $\frob*M$ and call it the \emph\fr\ of $M$; its elements are denoted by $*m$, and its $R$-module structure is then given by $r{*}m:=*r^qm$. If $k$ is perfect, then $R$ is \emph{F-finite},  and hence every \fr\  $\frob*M$ is again finitely generated. If we endow $\frob*R$ with the multiplication it inherits from $R$, namely $(*a)\cdot(*b):=*ab$, then it becomes an $R$-algebra, and its structure \homo\ is just another incarnation of the  Frobenius \homo. Our second result, in arbitrary dimension,  is:

\begin{proposition}\label{P:dirsumposdep}
If  there exists a direct sum decomposition  $\frob*Q\iso Q\oplus M$, for some   $R$-modules $Q$  and $M$, then $M$ satisfies the condition in  \Thm{T:pcanMCM}, that is to say,   $\can M1$ has positive depth.
\end{proposition}

Recall that $R$ is called \emph{F-pure} if $\frob q$ (whence $\frob p$) is pure, which  in the complete case, is equivalent with $R$ being a direct summand of $\frob*R$.\footnote{\label{f:Enescu}I am grateful to Florian Enescu for providing me the following argument for the converse:  suppose $s\colon R\to \frob*R$ is the embedding as a direct summand, and let $s(1):=*c$; then $s$ is the composition of the (Frobenius) \homo\ $R\to \frob*R$ followed by multiplication with $*c$, and as $s$ is split, so is then the former \homo, which is the definition of being F-split.}
We substantially weaken this condition by calling $R$ \emph{weakly F-split}, if it is F-finite and there exists some module $Q$ which is a direct summand of its \fr\ $\frob*Q$ (we will then call such a $Q$ \emph{F-split}; but be aware that there is in general no longer a morphism from $Q$ to its \fr\ $\frob *Q$). So we proved  (note that there are plenty of non-\CM\ F-pure rings):

\begin{maintheorem}\label{T:wFpureMCM}
In dimension three, any complete weakly F-split ring, whence in particular any complete F-pure ring, admits a small MCM.\qed
 \end{maintheorem} 
 
\subsection{Proof of \Prop{P:dirsumposdep}}
The proof requires two ingredients.  Firstly, we prove (\S\ref{s:frpcan}) that in general, we have an isomorphism
\begin{equation}\label{eq:frpcan}
\frob*(\can Mi)\iso \can{\frob*M}i.
\end{equation} 
Note, however, that this is a non-canonical isomorphism, which therefore requires some work. The second ingredient is the behavior of ordinal length under \fr{s}. However, for the  case at hand, we do not need to rely on ordinal length (a transfinite version of ordinary length studied in \cite{SchOrdLen}), and so we can make the following ad hoc observations: let $\mul M0$ be the length of zero-th local cohomology module $\hlc M0\maxim$. In particular, $M$ has positive depth \iff\ $\mul M0=0$. We show (\Cor{C:lenfr}) that if $k$ is perfect, then 
\begin{equation}\label{eq:lenfr}
\mul M0=\mul{\frob*M}0.
\end{equation} 
Now define $h(M):=\mul{\can M1}0$, so that the condition in \Thm{T:pcanMCM} is equivalent with $h(M)=0$. Suppose now that $R$ is weakly F-split, witnessed by an F-split module $Q$, that is to say, a decomposition $\frob*Q\iso Q\oplus M$. Since depth is preserved under faithfully flat descent, we can make a base change (a scalar extension $R\to \scal R{k^{1/p^\infty}}$, with $k^{1/p^\infty}$ the perfect hull of the residue field $k$,  in the sense of \cite[\S3]{SchClassSing}) so that the residue field becomes perfect. 
Since $h$ is additive on direct sums, we get $h(\frob*Q)=h(Q)+h(M)$. On the other hand, we have
\begin{equation}\label{eq:hpcan}
h(\frob*Q)=\mul{\can{\frob*Q}1}0\overset{\eqref{eq:frpcan}}=\mul{\frob*(\can Q1)}0\overset{\eqref{eq:lenfr}}=\mul{\can Q1}0=h(Q)
\end{equation} 
from which it follows that $h(M)=0$.\qed

We will prove \Thm{T:pcanMCM} in \S\ref{s:pcan}, whereas sections \S\ref{s:lenfr} and \S\ref{s:frcan} are then devoted to proving respectively \eqref{eq:lenfr} and \eqref{eq:frpcan}, completing the proof of our main theorem.  In \S\ref{s:examp}, we give a few examples (without proof) of F-split modules. However, as calculations blow up fast, we can only give examples over hypersurfaces, which are of course already \CM. We should point out that our use of Frobenius to obtain small MCM's  is different from  the graded case  as explained in \cite{HoInt} (for a proof, see \cite{HoCurr}), and also different from the toric case  \cite{SchToric}.

\section{Higher \pcan\ modules}\label{s:pcan}
In this section,  we fix a complete local ring $(R,\maxim)$ of dimension $d$. We also need to work occasionally with non-finitely generated (aka \emph{big}) modules and to emphasize this, we will denote them by capital Greek letters. 
Let $k$ be the residue field of $R$ and $E$ the injective hull of $k$. Recall that the \emph{Matlis dual} of a (big) $R$-module $\Omega$ is given by
$$
\matlis \Omega:=\hom R\Omega E
$$
It is an exact, contravariant functor, sending Noetherian modules to Artinian ones, and vice versa, and these are then canonically isomorphic to their biduals    (see, for instance, \cite[\S18]{Mats}). In particular, the Matlis dual of a module of finite length has again finite length, and in fact, the same length.

\begin{definition}\label{D:pcan}
Given a (finitely generated) $R$-module $M$, let us define its \emph{$i$-th \pcan\ module}, for $i=\range 0d$, by the rule 
$$
\can  Mi:=\matlis{\hlc M{d-i}\maxim }.
$$ 
\end{definition}
If we want to emphasize the base ring $R$, we will denote  these modules by $\canring MiR$. 
If   $i=0$, we just write $\can M{}$ for $\can  M0$.  Since each $\hlc M\bullet\maxim$ is Artinian, the $\can Mi$ are all finitely generated $R$-modules. By Grothendieck vanishing, if $s$ and $r$ are the respective depth and dimension of $M$, then $d-r$ and $d-s$ are respectively  the smallest and largest index $i$ for which $\can Mi\neq 0$. In particular, $M$ is an MCM \iff\ $\can Mi=0$, for all $i\geq 1$. 
The following fact will help us in studying \pcan\ modules, as it often allows us to reduce to the \CM\ case.

\begin{proposition}\label{P:finMatlisdual}
Let $\varphi\colon (S,\mathfrak n)\to (R,\maxim)$ be a finite morphism of complete local rings of relative dimension $r:=\dim S-\dim R$. Let $\Omega$ be an $R$-module, and let $\varphi_*\Omega$ be its pull-back along $\varphi$, that is to say, viewed as an $S$-module.
\begin{enumerate}
\item\label{i:finMat}  The Matlis $R$-dual of   $\Omega$  is   isomorphic over $S$ to the Matlis $S$-dual of $\varphi_*\Omega$, and $\hlc \Omega i\maxim\iso \hlc{\varphi_*\Omega}i{\mathfrak n} $ as $S$-modules, for all $i$. 
 \item\label{i:finpcan} For all $i$, we have an isomorphism of $S$-modules $\canring \Omega iR\iso\canring{\varphi_*\Omega}{r+i}S$.
\item\label{i:finCM} There always exists at least one local Gorenstein (whence \CM) ring $S$ and a finite morphism $\varphi\colon S\to R$ of relative dimension zero, and hence in this situation we have isomorphisms of $S$-modules 
$$\canring \Omega iR\iso\canring{\varphi_*\Omega}{i}S.$$ 
\end{enumerate}
\end{proposition} 
\begin{proof}
 Let $E_S$ and $E_R$ be the injective hull of the respective residue fields of $S$ and $R$, and let  $\matlisring R\Omega:=\hom R\Omega{E_R}$ and $\matlisring S{\Omega}:=\hom S{\varphi_*\Omega}{E_S}$   be the respective Matlis duals of $\Omega$ and its pull-back.  Since $\hom SR{E_S}$ is injective as an $R$-module by \cite[Lemma 3.1.6]{BH} and is supported only at $\maxim$, it must be a power of $E_R$, say $\hom SR{E_S}\iso E_R^e$, for some $e\in\nat$. Hence we have isomorphisms of $R$-modules
 $$
 (\matlisring R \Omega)^e=\hom R\Omega{E_R}^e\iso \hom R\Omega{\hom SR{E_S}}\iso \hom S{\varphi_*\Omega}{E_S}=\matlisring S{\Omega}.
 $$
So remains to show that $e=1$. Let $k$ be the residue field of $R$ and let $n$ be the length of its pull-back $\varphi_*k$ viewed as an $S$-module. As an $S$-module, the Matlis dual $\matlisring S{k}$   has the same length $n$, and is isomorphic to $(\matlisring R k)^e$ by the above. Since $k\iso \matlisring R k$, the   length as an $S$-module of $\varphi_*(\matlisring R k)^e$   is equal to $en$, so that $n=en$, whence $e=1$.

The analogue statement for local cohomology is well-known (see, for instance, \cite[\S3.5(3)]{BH}), so that combining these two facts proves \eqref{i:finpcan}. As for the last assertion, by Cohen's structure theorems, we can find  a regular local ring $T$ such that $R\iso T/I$ for some height $r$ ideal $I\sub T$. Since $T$ is \CM, we can find  a regular sequence $\rij xh$ inside $ I$. In particular, $S:=T/\rij xhT$ is Gorenstein  and $S\to R$ has relative dimension zero. 
\end{proof}

If $R$ is \CM, then $\can R{}$ is just its canonical module by Grothendieck duality.  Without the \CM\ assumption, we only have 
\begin{equation}\label{eq:toppcan}
\can M{}\iso \hom{}M{\can R{}},
\end{equation} 
which follows from applying Matlis duality to the isomorphism $M\tensor \hlc Rd\maxim\iso \hlc Md\maxim$. If $S\to R$ is as in \eqref{i:finCM}, then $\can R{}\iso\hom SRS$, since $S$ is then its own canonical module. Nonetheless, the \pcan\ modules still form a contravariant, additive $\delta$-functor in the sense of Grothendieck \cite{GrotHom}, that is to say:

\begin{proposition}\label{P:deltapcan}
Given an exact sequence $\Exactseq NMQ$, we get  canonically defined transition maps $\partial_i\colon \can Ni\to \can Q{i+1}$ that fit into a long exact sequence  
\begin{multline*}
0\to \can Q0\to\can M0\to\can N0\map{\partial_0}\can Q1\to\can M1\to 
\dots \\ \dots\to\can N{d-1}\map{\partial_{d-1}}\can Qd \to \can Md\to\can Nd\to 0.
\end{multline*}
\end{proposition} 
\begin{proof}
Immediate from the long exact sequence of local cohomology and the exactness of Matlis duality.
\end{proof}
\begin{remark}\label{R:deltapcan}
Note, however, that the $\can \cdot\bullet$ are not derived functors if $R$ is not \CM. Put differently, we do no longer have Grothendieck duality: in general $\can Mi$ is different from $\ext RiM{\can R{}}$; see \Prop{P:dimpcan} below.
\end{remark}

Let us call an element $m$  in a module $M$ \emph{small}, if $\dim{R/\ann{}m}$, that is to say the dimension of the module generated by $m$, is strictly smaller than the dimension of $M$ itself. The subset of all small elements forms a submodule, denoted $\fl s (M)$, and the resulting quotient $\unmix M:= M/\fl s(M)$ is called the \emph{unmixed quotient} of $M$. Let us call a module \emph{unmixed}, if $\fl s(M)=0$. It is easy to see that $\unmix M$ is unmixed, and in fact, it is the largest unmixed quotient of $M$.

\begin{proposition}\label{P:dimpcan}
Let $d:=\dim R$ and $M$ a $d$-dimensional module. Then $\can M{}$ is unmixed of dimension $d$ and $\dim{\can Mi}\leq d-i$, for all $i\geq 1$. Moreover, if $S\to R$ is finite of relative dimension zero and $S$ is \CM, then we have natural isomorphisms of $S$-modules
\begin{equation}\label{eq:Grdualpcan}
\can Mi\iso \ext {S}iM{\can R{}}
\end{equation} 
for all $i$, where for simplicity, we also wrote $M$ for its pull-back as an $S$-module. 
\end{proposition} 
\begin{proof}
By \eqref{i:finpcan}, since the relative dimension is zero, we may view the $\can Mi$ as the \pcan\ $S$-modules of the pull-back (justifying     reference omission to the base ring). In particular, as such, $\can R{}$ is just the canonical module of $S$, and the isomorphisms~\eqref{eq:Grdualpcan} are now just Grothendieck duality over $S$ (see also \cite[Exercise 3.5.14]{BH}. Moreover, by \eqref{i:finCM}, a morphism $S\to R$ as above always exist. It is well-known (\cite[Corollary 3.5.11]{BH}) that $\ext {S}iM{\can R{}}$ has dimension at most $d-i$, whence our last claim (note that the dimension of a module does not depend on the base ring). To prove the first assertion, by \eqref{eq:toppcan}, we only need to show that $\can R{}$ is unmixed as an $R$-module, but this is the same as being unmixed as an $S$-module, which follows since the latter is the canonical module of $S$. 
%
%
\end{proof}

\begin{lemma}\label{L:pcanunm}
For any $R$-module $M$, we have a canonical isomorphism $\can M{}\iso \can{\unmix M}{}$ and a natural embedding $\can {\unmix M}1\sub\can M1$. 
\end{lemma} 
\begin{proof}
By \Prop{P:deltapcan},   the exact sequence $\Exactseq {\fl s(M)}M{\unmix M}$ gives rise to a long exact sequence
$$
0\to \can{\unmix M}{}\to \can M{}\to \can{\fl s(M)}{}\to \can{\unmix M}1\to\can M1
$$
By Grothendieck vanishing, $\can{\fl s(M)}{}=0$, and the assertion follows.
\end{proof} 

Our first application is an abundant source of MCM's in dimension two:

\begin{corollary}\label{C:dim2}
If $R$ is a two-dimensional complete local ring and $M$ any two-dimensional $R$-module, then $\can M{}$ is an MCM.
\end{corollary} 
\begin{proof}
Put $Q:=\can M{}$, and since this is isomorphic to $\can{\unmix M}{}$ by \Lem{L:pcanunm}, we may already assume from the start that $M$  is unmixed. 
Let $x$ be a parameter (=element outside all maximal dimensional prime ideals) on $R$. By unmixedness,  it is $M$-regular. Put $\bar M:=M/xM$. The exact  sequence $\exactseq MxM{}{\bar M}$ yields, by \Prop{P:deltapcan} and Grothendieck vanishing,   a long exact sequence
$$
0=\can{\bar M}0\to Q\map x Q\to \can {\bar M}1\to\dots
$$
so that in particular, $Q/xQ\sub \can {\bar M}1$. If we let $\bar R:=R/xR$, then by \eqref{i:finpcan}, we have $\can{\bar M}1=\canring{\bar M}1R=\canring {\bar M}0{\bar R}$. By \Prop{P:dimpcan}, the latter is unmixed (as an $\bar R$-module), whence has depth at least one. Hence  $Q/xQ$, being a submodule,   has  also depth at least one. Since $x$ is  $Q$-regular, as $Q$   is unmixed of dimension two by \Prop{P:dimpcan},  we showed that $Q$ has depth at least two,  i.e., is MCM.
\end{proof}
\begin{remark}\label{R:dim2}
We actually proved that if $\dim M=\dim R\geq 2$, then $\can M{}$ has depth at least two.
\end{remark} 

\subsection{Proof of \Thm{T:pcanMCM}}
If $M$ is an MCM, then  $\can M1=0$ by Grothendieck vanishing (and the zero module has by definition infinite depth). For the converse, assume $M$ is a three-dimensional module such that $\can M1$ has positive depth.  Since $\can {\unmix M}1$ is a submodule of the latter by \Lem{L:pcanunm}, it too has positive depth, and so we may assume from the start that $M$ is unmixed. By assumption, we can find a parameter $x$ which is $\can M1$-regular. By unmixedness, it is also $M$-regular. Put $\bar R:=R/xR$ and $\bar M:=M/xM$. From the short exact sequence $\exactseq MxM{}{\bar M}$ we get by \Prop{P:deltapcan} and Grothendieck vanishing,   a long exact sequence
$$
0=\can{\bar M}{}\to \can M{}\map x \can M{}\to \can {\bar M}1\to \can M1\map x\can M1
$$
By assumption the latter map is injective, showing that $\can{\bar M}1\iso \can M{}/x\can M{}$. By \eqref{i:finpcan}, we have $\can{\bar M}1=\canring{\bar M}0{\bar R}$, and by \Cor{C:dim2}, this is an MCM over $\bar R$, that is to say, has depth two. Hence $\can M{}$ has depth three, whence is an MCM over $R$.\qed

\begin{remark}\label{R:pcanMCM}
Our argument actually shows that if $R$ is a $d$-dimensional complete local ring and $M$ a  $d$-dimensional $R$-module  such that $\can M1$ has positive depth, then $\can M{}$ has depth at least three.
\end{remark}

\section{The \fr\ of a module}\label{s:lenfr}

In this section,   $(R,\maxim)$  will always denote a Noetherian local ring of \ch\ $p>0$, with residue field $k$. Moreover, $q$ will always denote some power of $p$. We define its \emph{$q$-th degree of imperfection}  as $\iota_q(R):=(\frob*k:k)$,  which is the same as the (vector space) degree of $k$ over its subfield $k^q$ of all $q$-th powers, and also the same as  the degree of   the field of $q$-th roots $k^{1/q}$ over $k$. In particular, $k$ is perfect \iff\ $\iota_p(R)=1$. A standing assumption, moreover, will be that $R$ is \emph{F-finite}, meaning that the Frobenius is a finite morphism. In particular, all degrees of imperfection $\iota_q(R)$ will be finite, and the converse holds when $R$ is complete.  

We defined in the introduction the \fr\ functor $\frob{q*}$ on the category of $R$-modules (the F-finiteness assumption implies that $\frob*M$ is again finitely generated). Whenever  $q$ is clear from the context, we just denote it by $\frob*$. (A note of caution, do not confuse this with the    Peskine-Szpiro Frobenius functor  given by $\mathfrak F(M):=M\tensor_R\frob*R$.) Recall that $\frob*M$ is the $R$-module whose  elements are denoted by $*m$, for $m\in M$, and with the scalar action of $R$ given by $r{*}m:=*r^pm$.  An easy, but very important property of the \fr\ is its exactness.

 \begin{lemma}\label{L:locFrob}
Given a multiplicative set $\Sigma$, we have $\frob{*}(\inv \Sigma M)\iso \inv \Sigma (\frob{*}M)$.
\end{lemma} 
\begin{proof}
Any element of $\frob{*}(\inv \Sigma M)$ is of the form ${*}(\tfrac ms)$, with $m\in M$ and $s\in \Sigma $. To it, we let correspond the element $\tfrac{{*}(s^{p-1}m)}s$ in $\inv \Sigma (\frob{*}M)$. To see that this is well defined, suppose $\tfrac ms=\tfrac{\tilde m}{\tilde s}$ in $\inv \Sigma M$, for some $\tilde m\in M$ and $\tilde s\in \Sigma $. Hence, $\tilde stm=st\tilde m$ in $M$, for some $t\in \Sigma $. Multiplying both sides with $(s\tilde st)^{p-1}$, we get $s^{p-1}\tilde s^pt^pm=s^p\tilde s^{p-1}t^p\tilde m$ in $M$, whence $\tilde st{*}(\tilde s^{p-1}m)=st{*}( s^{p-1}\tilde m)$ in $\frob{*}M$, showing that $\tfrac{{*}(s^{p-1}m)}s=\tfrac{{*}(\tilde s^{p-1}\tilde m)}{\tilde s}$ in $\inv \Sigma (\frob{*}M)$. The converse is defined by sending $\tfrac{{*}n}t$, with $n\in M$ and $t\in \Sigma $ to ${*}(\tfrac n{t^p})$, which again, by a similar argument, is well-defined. Moreover, the compositions   ${*}(\tfrac ms)\mapsto \tfrac{{*}(s^{p-1}m)}s\mapsto {*}(\tfrac{s^{p-1}m}{s^p})={*}(\tfrac ms)$ and $\tfrac{{*}n}t\mapsto {*}(\tfrac n{t^p})\mapsto \tfrac{{*}(t^{p(p-1)}n}{t^p}=\tfrac{t^{p-1}{*}n}{t^p}=\tfrac{{*}n}t$ are the identity, showing that these maps are each others inverse. We leave it to the reader to verify that these maps are $R$-linear, and therefore give the desired isomorphism.
\end{proof}

\begin{theorem}\label{T:loccohFrob}
For each $R$-module $M$ and each $i\geq 0$, we have a canonical isomorphism
$$
\frob{*}\hlc Mi\maxim\iso \hlc{\frob{*}M}i\maxim.
$$
\end{theorem} 
\begin{proof}
Let $\rij xd$ be a system of parameters in $R$ and let $C_M^\bullet$ be the corresponding \v Cech complex on $M$, that is to say, for each $i$, the module $C_M^i$ is the direct sum of all localizations $M_z$ where $z$ runs over all products of $i$ distinct elements from the system of parameters, and the differential $C_M^i\to C_M^{i+1}$ is given by the natural localization maps, up to a sign (for details, see \cite[\S3.5]{BH}). In particular, $C_M^\bullet\iso C_R^\bullet\tensor_RM$, and the homology of this complex is the local cohomology $H_\bullet(C_M^\bullet)=\hlc M\bullet\maxim$. By \Lem{L:locFrob}, we have an isomorphism of complexes $\frob{*}(C_M^\bullet)\iso C_{\frob{*}M}^\bullet$. Taking homology and using that $\frob {*}$, being exact, commutes with homology, we get
$$
\frob{*}(\hlc M\bullet\maxim)=\frob{*}(H_\bullet(C_M^\bullet))\iso H_\bullet(\frob{*}C_M^\bullet)\iso H_\bullet(C_{\frob{*}M}^\bullet)=\hlc{\frob{*}M}\bullet\maxim.
$$
\end{proof} 

Immediate form this, we get:
 \begin{corollary}\label{C:frMCM}
If $M$ is an MCM, then so is its \fr\ $\frob*M$.\qed
\end{corollary}

\begin{proposition}\label{P:finlenFrob}
If $M$ has finite length, then $\len{\frob{*}M}=\iota\cdot\len M$, where $\iota:=\iota_q(R)$.
\end{proposition} 
\begin{proof}
We want to show by induction on $l:=\ell(M)$ that $\frob {*}M$ has length $l\iota  $. When $l=1$, then $M\iso k$, and the result holds by definition. For $l>1$, choose an exact sequence $\Exactseq NMk$ with $\len N=l-1$. Since $\frob {*}$ is an exact functor, we get an exact sequence $\Exactseq{\frob{*}N}{\frob{*}M}{\frob{*}k}$, showing that $\len{\frob{*}M}=\len{\frob{*}N}+\len{\frob{*}k}=(l-1)\iota+\iota=l\iota $, as we wanted to show. 
%
%
\end{proof} 

\begin{corollary}\label{C:finlenitFrob}
For any module $M$ of finite length $l$ over a local ring $R$ with perfect residue field $k$, we have $\frob{*}^nM\iso k^l$, for all $n\gg0$.
\end{corollary} 
\begin{proof}
Since $\maxim^l$ annihilates $M$, an easy induction argument yields that $\maxim$ annihilates $\frob{*}^lM$. Hence $\frob{*}^lM$ is a vector space,  of length $l$ by \Prop{P:finlenFrob}.
\end{proof} 

Recall that  $\mul M0:=\len{\hlc M0\maxim}$, which is therefore the largest length of an Artinian submodule of $M$. We can now give a more general form of \eqref{eq:lenfr}:

\begin{corollary}\label{C:lenfr}
  With $\iota:=\iota_q(R)$, we have, for any finitely generated $R$-module,  
  \begin{equation}\label{eq:lenfrimp}
\mul{\frob*M}0=\iota \mul M0.
\end{equation} 
\end{corollary} 
\begin{proof}
Let $H:=\hlc M0\maxim$ be the maximal Artinian submodule of $M$. It is not hard to see that $\frob*H$ is then the maximal Artinian submodule of $\frob*M$. The result follows, since  $\len{\frob*H}=\iota\len H$ by \Prop{P:finlenFrob}. 
\end{proof}

\section{The \fr\ of  a \pcan\ module}\label{s:frcan}
We will interpret $m$-tuples $\tuple x$ as row vectors and so, their image under an $m\times n$-matrix $\mat A$ is given by $\tuple x\mat A$.  For our purposes, it is more convenient to use the logicians' numbering, that is to say, we list the entries $a_{ij}$ of $\mat A$ from $i=\range 0{m-1}$ and $j=\range 0{n-1}$. 
%

Given a morphism $f\colon M\to N$, the induced morphism $\frob*f\colon\frob*M\to \frob*N$ is given by the rule $*m\mapsto *f(m)$. Let $\mulel rM$ be the endomorphism on $M$ given by multiplication  of a scalar $r\in R$. Note that $\frob*(\mulel rM)$ is in general not multiplication with $r$, but, in fact
$$
(\frob*(\mulel rM))^q=\mulel {r}{\frob*M}.
$$
For any module $M$, we obtain a ring \homo\footnote{Note that this is not a morphism of $R$-algebras; to make it into one, we should instead consider the morphism $\frob{*}R\to \ndo{\frob{*}M}\colon {*}r\mapsto \frob{*}(\mulel rM)$, but this will be of lesser use to us.}
$$
\nabla_M\colon R\to \ndo{\frob{*}M}\colon r\mapsto \frob{*}(\mulel rM).
$$

 Given an $m\times n$-matrix  $\mat A$ and a module $M$, we    get a morphism $M^m\to M^n\colon \tuple m\mapsto \tuple m\mat A$, which we will denote similarly by $\mulel{\mat A}M$ (and when there is no possibility for confusion, we just write the matrix for the morphism it defines). However, for the converse,  a morphism $M^m\to M^n$ is given by a matrix with entries over  the  (possibly non-commutative) endomorphism ring $\ndo M$.  The \fr\ $\frob{*}(\mulel{\mat A}M)$ is given by  the rule $\tuple m\mapsto {*}\tuple m\mat A$, and its corresponding   matrix has coefficients in  $\ndo {\frob{*}M}$, namely:

\begin{lemma}\label{L:frobmat}
Given an $m\times n$-matrix $\mat A$ over $R$, we will write $\mat A^{\nabla_M}$ for the  matrix   over $\ndo{\frob{*}M}$ obtained from $\mat A$ by applying $\nabla_M\colon R\to \ndo{\frob{*}M}$ to each of its entries. Then 
$$
\frob{*}(\mulel{\mat A}M)=\mulel{(\mat A^{\nabla_M})}{\frob{*}M}.\qed
$$
\end{lemma} 

\subsection*{Quasi-symmetric matrices}
We need some terminology from linear algebra. The $n\times n$-identity matrix will be denoted by $\mat I_n$, or just $\mat I$, if its size is clear from the context. We also need the $n\times n$-\emph{exchange matrix} $\mat J_n$ (or just $\mat J$), which is the matrix with ones on the counterdiagonal (=SE/NW diagonal) and zeros elsewhere. Note that $\mat J^2=\mat I$, and in particular, $\mat J$ is its own inverse. Fix $n$ and let $\mat A=(a_{ij})$ be a square $n\times n$-matrix. 
%
%
%
%
Recall that $\mat A$ is called \emph{symmetric} if it is equal to its own transpose $\trans A$; in matrix entries, this means $a_{ij}=a_{ji}$. We call $\mat A$   \emph{\persym}, if it is symmetric around its counterdiagonal. This is equivalent with $\mat A\mat J=\mat J\trans A$. 
In terms of its entries, this means,  
\begin{equation}\label{eq:qsymm}
a_{ij}=a_{n-1-j,n-1-i}.
\end{equation} 

\subsection*{The \fr\ over a regular local ring}
Let $S$ be a $d$-dimensional complete, regular local ring of \ch\ $p$ with perfect residue field $k$. By Kunz's theorem, $\frob*S$ is a free $S$-module of rank $n:=q^d$. More concretely, $S\iso \pow k{\tuple x}$, with $\tuple x=(x_0,\dots,x_{d-1})$. For $a=\range 0{q^d-1}$, let 
  $\qdig ak$ be its $q$-adic digits, that is to say, $a=\sum \qdig akq^k$ is the $q$-adic expansion of $a$, with $0\leq \qdig ak<q$, and put  
  $$
\tuple m_a:=x_0^{\qdig a0}x_1^{\qdig a1}\cdots x_{d-1}^{\qdig a{d-1}}
$$
  The $*\tuple m_a$, for $a=\range 0{q^d-1}$, then form a basis of $\frob* S$ over $S$, called the standard basis and we will consider all our matrices with respect to this basis (in the given order). 
%
Given $s\in S$, we want to describe the endomorphism $\nabla_S(s)=\frob*(\mulel sS)$ on $\frob*S$. Let us denote its matrix with respect to the standard basis by $\mat D_S(s)$ or just $\mat D(s)$.

\begin{proposition}\label{P:multskewsymm}
Each matrix $\mat D(s)$ is \persym{}.
\end{proposition} 
\begin{proof}
Any linear combination of \persym{} matrices is again \persym{}. Since $S$ is generated by its monomials, it suffices therefore to show the claim for $s=\tuple x^\gamma$ a monomial, with $\gamma=(c_0,\dots,c_{d-1})\in\nat^d$. 
Let $s_{ab}$, for $a,b<q^d$, be the entries of $\mat D(s)$. For an arbitrary integer $c$, let $\quot c$ and $\rem c$ be its respective quotient and remainder after division by $q$, so that $c=q\quot c+\rem c$. Define a permutation, denoted again $s$, on the indices by the rule
$$
s(a):= \sum_{k=0}^{d-1}\rem{c_k+\qdig ak}q^k
$$
Hence,  
$$
\begin{aligned}
\frob*(\mulel sS)(\tuple m_a)&=*\tuple x^\gamma\tuple m_a= \prod_{k=0}^{d-1}x_k^{\quot{c_k+\qdig ak}}{*}\prod_{k=0}^{d-1}x_k^{\rem{c_k+\qdig ak}}\\
&= \prod_{k=0}^{d-1}x_k^{\quot{c_k+\qdig ak}}{*}\tuple m_{s(a)}
\end{aligned}
$$
Hence $s_{ab}$ is zero, unless $b=s(a)$, in which case it is equal to 
\begin{equation}\label{eq:qcak}
s_{ab}=\prod_{k=0}^{d-1}x_k^{\quot{c_k+\qdig ak}}.
\end{equation} 
 We need to verify identity \eqref{eq:qsymm}. From
$$
q^d-1-a=\sum_{k=0}^{d-1}(q-1-\qdig ak)q^k
$$
we read off its $q$-adic digits as $q-1-\qdig ak$. Hence, applied to $b$, we get 
$$
s(q^d-1-b)=\sum_{k=0}^{d-1}\rem{c_k+q-1-\qdig bk}q^k
$$
and this equal to $q^d-1-a$  \iff\ $\rem{c_k+q-1-\qdig bk}=q-1-\qdig ak$. The latter equality is the same as the equivalence $c_k-\qdig bk\equiv -\qdig ak\mod q$ whence $c_k+\qdig ak\equiv  \qdig bk\mod q$, which in turn is equivalent with $s(a)=b$. Therefore, if $a\neq s(b)$, then $s_{q^d-1-b,q^d-1-a}=0$, and so we only need to calculate it in the case that   $a=s(b)$. In that case, 
\begin{equation}\label{eq:qcbk}
s_{q^d-1-b,q^d-1-a}=\prod_{k=0}^{d-1}x_k^{\quot{c_k+q-1-\qdig bk}}.
\end{equation} 
Since $c_k+\qdig ak\equiv  \qdig bk\mod q$, we can write 
\begin{equation}\label{eq:cabk}
c_k+\qdig ak=qv_k+\qdig bk,
\end{equation} 
 for some $v_k$. Since $\qdig bk<q$, we see that $v_k= \quot{c_k+\qdig ak}$. On the other hand, using \eqref{eq:cabk}, we see that $c_k+q-1-\qdig bk=qv_k+q-1-\qdig ak$, and since $q-1-\qdig ak<q$, this shows that $v_k=\quot{c_k+q-1-\qdig bk}$, so that \eqref{eq:qcak} and \eqref{eq:qcbk} are the same, as we wanted to show. 
%
%
%
\end{proof} 

Let $E$ be the injective hull of $k$. 

\begin{proposition}\label{P:GorinjFrob}
We have a canonical isomorphism $\frob*E\iso E\tensor\frob*S$.  
\end{proposition} 
\begin{proof}
Since $S$ is in particular Gorenstein, whence equal to its own   canonical module,    $\hlc Sd\maxim\iso E$ by Grothendieck duality.  Using \Thm{T:loccohFrob}, we get  canonical isomorphisms
$$
\frob*E=\frob*\hlc Sd\maxim=\hlc{\frob*S}d\maxim\iso \hlc Sd\maxim\tensor\frob*S\iso E\tensor\frob*S.
$$
In fact, representing $E=\hlc Sd\maxim$ as the cokernel of the Cech complex $C_S^{d-1}\to C^d_S=S_x$, where $x=x_1\cdots x_d$ for a fixed system of parameters $\rij xd$, and writing $\class{\tfrac a{x^n}}$ for the image of $a/x^n$ in $\hlc Sd\maxim$, we can trace the above isomorphism explicitly, and show that it is given by
\begin{equation}\label{eq:injFrob}
E\tensor \frob*S\to \frob*E\colon \class{\tfrac a{x^n}}\tensor *r\mapsto *\class{\tfrac{ra^p}{x^{np}}}
\end{equation} 
To define the converse, observe that any element in $E$ is of the form $\class{\tfrac a{x^{pn}}}$, by multiplying numerator and denominator with a suitable power of $x$. The corresponding element $*\class{\tfrac a{x^{pn}}}$  is then sent under the inverse isomorphism to $\class{\tfrac 1{x^n}}\tensor *a$.
\end{proof} 
\begin{remark}\label{R:GorinjFrob}
Note that we only used the fact that $E\iso \hlc Sd\maxim$,  and so the proof works in fact for any  quasi-Gorenstein ring $S$.
\end{remark} 

Since $\frob*S\iso S^{q^d}$, we get in fact $\frob*E\iso E^{q^d}$. In particular, since $\ndo E=S$, the \fr\ of multiplication with an element $s\in S$ on $E$, that is to say, $\frob*(\mulel sE)$, is given by a $q^d\times q^d$-matrix with coefficients in $S$, which we will denote by $\mat D_E(s)$.

\begin{lemma}\label{L:matmulinjhull}
For any element $s\in S$, we have an equality of matrices 
$\mat D_E(s)=\mat D_S(s)$.
\end{lemma} 
\begin{proof}
Let $s_{ab}$ be the entries of $\mat D_S(s)$. 
  The $(a,b)$-th entry of $\mat D_E(s)$ is the endomorphism given by the composition
$$
E\map {i_a} E\tensor \frob*S\iso \frob*E \map{\frob*(\mulel{s}{E})}\frob*E\iso E\tensor\frob*S\map{\pi_b} E
$$
where the isomorphisms are given by \eqref{eq:injFrob}, and $i_a$  (respectively, $\pi_b$) is the base change  of the $a$-th embedding (respectively $b$-th projection map) of $S$ into $\frob*S$ (respectively, of $\frob*S$ onto $S$). Since $\ndo E=S$, this endomorphism is then given by multiplication with an element, and we need to show that this is $s_{ab}$. To this end, take $z\in E$. In the notation of the previous proof, it is of the form $z=\class{\tfrac v{x^n}}$, with $v\in S$, $n\in\nat$, and $x:=x_0\cdots x_{d-1}$. Its image under the above composition is 
$$
z\overset{i_a}\mapsto z\tensor *\tuple m_a\overset{\eqref{eq:injFrob}}\mapsto *\class{\tfrac{v^q\tuple m_a}{x^{nq}}}\overset{\frob*(\mulel{s}{E})}\mapsto *\class{\tfrac{v^qs\tuple m_a}{x^{nq}}}\overset{\eqref{eq:injFrob}}\mapsto \class{\tfrac v{x^n}}\tensor *s\tuple m_a\overset{\pi_b}\mapsto  s_{ab}z 
$$
since $*s\tuple m_a=\sum_b s_{ab}{*}\tuple m_b$, proving the claim.
\end{proof}

To prove \eqref{eq:frpcan}, we temporarily  introduce the following functor on the category of finitely generated modules
\begin{equation}\label{eq:frobmatlis}
\frobmatlis M:=\matlis{(\frob{*}(\matlis M))}
\end{equation} 

 \begin{theorem}\label{T:matlisFrobfuncreg}
Over a complete  local ring $R$,  we have $\frob{*}(M)\iso \frobmatlis M$, for any $R$-module $M$.
\end{theorem} 
\begin{proof}
Assume first that $R$ is regular of dimension $d$, so that $\frob{*}R\iso R^{q^d}$ by Kunz's Theorem. Using   \Prop{P:GorinjFrob}, we get
$$
\frobmatlis R=\matlis{(\frob*E)}\iso \matlis{(E^{q^d})}=R^{q^d}\iso\frob*R.
$$
%
 For $M$ an arbitrary finitely generated $R$-module, we can represent it as a the cokernel of a (square) matrix $\mat A$, that is to say, by an exact sequence $ R^n\map {\mat A} R^n\to M\to 0$. By \Lem{L:frobmat}, its \fr\  is 
\begin{equation}\label{eq:frobf}
  \frob{*}R^n\map{\mat A^{\nabla_R}} \frob{*}R^n\to \frob*M\to 0
\end{equation} 
Recall that $\mat A^{\nabla_R}$  is an $n\times n$-matrix of $(q^d\times q^d)$-matrices: if the $(i,j)$-th entry of $\mat A$ is $a_{ij}$, then the $(i,j)$-th entry of $(\trans  A)^{\nabla_R}$ is the   matrix $\mat D_R(a_{ji})$. On the other hand, by Matlis duality $0\to \matlis M\to E^n\to E^n$ is given by the transpose $\trans A$. Taking again the \fr\ and using \Lem{L:frobmat}, we get an exact sequence
 $$
0 \to \frob*(\matlis M)\to (\frob*E)^n\map{(\trans A)^{\nabla_E}} (\frob*E)^n
 $$
 and taking one more time Matlis duals, an exact sequence
\begin{equation}\label{eq:Tf}
 \frobmatlis R^n\map{\trans{((\trans A)^{\nabla_E})}}\frobmatlis R^n\to \frobmatlis M\to 0
\end{equation} 
%
%
 Let $\mat P$ be the matrix $\op{diag}(\mat J,\dots,\mat J)$, that is to say, the diagonal $n\times n$-matrix with diagonal  elements the  $q^d$-th  exchange matrix $\mat J$. Since each $\mat D_E(a_{ji})$ is \persym{} by \Prop{P:multskewsymm} and equal to $\mat D_R(a_{ji})$ by \Lem{L:matmulinjhull}, one easily verifies the following matrix relation
$$
\mat P\cdot(\trans A)^{\nabla_E}\cdot\inv{\mat P}=\trans{\left(\mat A^{\nabla_R}\right)}
$$ 
Hence, as $\frobmatlis R\iso\frob*R$ and  the matrices in \eqref{eq:frobf} and \eqref{eq:Tf} are conjugate, their cokernels are isomorphic, as we needed to show.
%

For $R$ arbitrary, by Cohen's structure theorem, we can find a regular local ring $S$ and a surjection $S\to R$. The functor $\frob{*}(\cdot)$ on the category of $R$-modules remains the same, if we consider them instead as $S$-modules. By \Prop{P:finMatlisdual}, the same is true for Matlis duality, whence also for the composite functor $\frobmatlis \cdot$. Hence, as $S$-modules, we have an isomorphism $\frob{*}M\iso \frobmatlis M$, which therefore is also an isomorphism over $R$.
\end{proof} 

\subsection{Proof of \eqref{eq:frpcan}}\label{s:frpcan}
For all $i$, we have
$$
\begin{aligned}
 \can{\frob*M}i &\overset{\ref{D:pcan}}\iso \matlis{(\hlc{\frob*M}{d-i}\maxim)} 
\overset{\ref{T:loccohFrob}}\iso \matlis{(\frob*\hlc M{d-i}\maxim)}\\
&\overset{\ref{D:pcan}}\iso \matlis{(\frob*(\matlis{\can Mi}))}
\overset{\eqref{eq:frobmatlis}}=\frobmatlis{\can Mi}
\overset{\ref{T:matlisFrobfuncreg}}\iso \frob*\can Mi.\qed
\end{aligned}
$$

We also derive the following dual version of Kunz's theorem:

\begin{theorem}\label{T:regfrobinj} 
A complete local ring $R$ with perfect residue field $k$ is regular \iff\ $\frob*E$ is injective, where $E$ is the injective hull of   $k$.
\end{theorem} 
\begin{proof}
One direction is immediate from \Prop{P:GorinjFrob}, so assume $\frob*E$ is injective, necessarily of the form $E^n$.   \Thm{T:matlisFrobfuncreg} then yields $\frob *R\iso \frobmatlis R=\matlis{(\frob*E)}=R^n$, so that by Kunz's theorem, $R$ is regular. 
\end{proof} 

\subsection*{F-trivialization}
We have already seen in \Cor{C:finlenitFrob} that \fr{s} trivialize modules of finite length. We now conjecture that a weaker version of this is true in general (we formulate it only for top dimension, as this is the only case we need, but presumably this holds for any non-simple module):

\begin{conjecture}\label{C:dirsumfrob}
A complete $d$-dimensional local ring $R$ in positive \ch\ is \emph{F-trivializing}, meaning that  for every $d$-dimensional $R$-module $M$, there is some $n$ such that $\frob{*}^nM$ is decomposable.
\end{conjecture} 

\begin{theorem}\label{T:dirsumfrobMCM}
An F-trivializing complete three-dimensional local ring $R$  admits an MCM.
\end{theorem} 
\begin{proof}
With $h$ as defined in the introduction (after \eqref{eq:lenfr}), suppose $R$ does not admit an MCM. Hence $h(M)>0$, for all three-dimensional modules $M$. Choose $M$ with $h(M)$ minimal. By assumption, we can find $q=p^n$ and a decomposition $\frob*M=Q\oplus N$. Using \eqref{eq:hpcan} and the additivity of $h$, we get  $h(M)=h(\frob*M)=h(Q)+h(N)$. But the latter sum violates the minimality of $h(M)$, contradiction.
\end{proof}

To verify F-trivialization, we must check all $d$-dimensional modules, and I do not even know whether this is true for  regular local rings of dimension $d\geq 2$. However, less is needed for the proof: let us call a class $\mathfrak C$ of $d$-dimensional modules an \emph{F-net}, if it is closed under \fr{s} and direct summands. We leave it to the reader to check that the class of unmixed modules is an F-net. We say that $R$ is \emph{weakly F-trivializing}, if there exists an F-net $\mathfrak C$ such that for each $M\in\mathfrak C$, there exists $n\in\nat $ for which $\frob*^nM$ is decomposable.   We may now weaken the assumption in \Thm{T:dirsumfrobMCM} to being weakly F-trivializing: indeed, in the proof, just take the minimal $h(M)$ for $M$ in the F-net. For a given module $Q$, let $\mathfrak F_Q$ be the collection of direct summands of its \fr{s} $\frob*^nQ$. If $Q$ is unmixed, then so is any module in $\mathfrak F_Q$, and so this is an example of an F-net; in fact, it is the smallest F-net generated by $Q$. A regular local ring is now easily seen to be weakly F-trivializing: indeed, $\mathfrak F_S$ is just the class of   free  modules, and so is clearly F-trivializing. We may restate this now as follows:

\begin{corollary}\label{C:wFtriv}
If $R$ is a complete three-dimensional local ring admitting an unmixed module $Q$ such that any direct summand of a \fr\ of $Q$ has itself a \fr\ that is decomposable, then $R$ admits an MCM.\qed
\end{corollary} 

In the examples below, the structure of the F-net $\mathfrak F_R$ appears already  a quite intriguing invariant of $R$. 
%
%

\section{Examples}\label{s:examp}

We will present most examples without the calculations (they become tedious  in higher dimensions, but explaining the methods to tackle these would be the topic of a paper on its own). Unfortunately, all examples  are   \CM, and so their weak F-splitness is rather a  moot point.

\begin{example}[Regular local rings]\label{E:reg2}
Let $(S,\maxim)$ be the power series ring in $d$ variables over a perfect field $k$ of \ch\ $3$. If $d=2$, we have
$$
\frob*S\iso S^9\qquad\text{and}\qquad \frob*\maxim\iso \maxim\oplus S^8
$$
showing that the F-net $\mathfrak F_\maxim$ consists of direct sums of copies of $R$ and $\maxim$.  Similarly, one can show that $\frob*\maxim^2\iso \maxim S^3\oplus S^6$ and $\frob{*}\maxim^3\iso  \maxim S^6\oplus S^3$. However, the structure of $\frob*\maxim^4$ seems no longer to follow this pattern as it contains a free summand of rank four. 

For $d=3$, after some lengthy calculations, we find $\frob*\maxim\iso \maxim\oplus S^{26}$. It seems therefore reasonable to postulate the same in arbitrary dimension and \ch
$$
\frob*\maxim \overset?\iso \maxim\oplus S^{p^d-1}.
$$
\end{example}

\begin{example}[Curves]\label{E:cusp}
Let $(R,\maxim)$ be the local ring at the origin of the cusp $x^2=y^3$ with $p=3$. Note that $R$ is not F-pure (for in dimension one this is equivalent with being regular). Indeed, one verifies that
$$
\frob*R\iso \frob*\maxim\iso \maxim R^3.
$$
In particular, $\maxim$ is F-split whence $R$ is weakly F-split. Moreover,  the F-net $\mathfrak F_R$ consists of all direct sums of copies of $R$ and $\maxim$, and $R$ is F-trivializing for this F-net, that is to say, $R$ is weakly F-trivializing.  

Take instead the local ring of the curve $x^3=y^4$ in \ch\ $p=5$. We have
$$
\frob*R\iso (x,y^2)R\oplus \maxim^2R^4\qquad\text{and}\qquad \frob*\maxim^2\iso\frob*(x,y^2)R\iso \maxim^2R^5.
$$
In particular, $R$ is weakly F-split witnessed by the F-split module $\maxim^2$, and also  weakly F-trivializing, witnessed by the F-net  $\mathfrak F_R$, which is generated by the three indecomposables $R$, $(x,y^2)R$ and $\maxim^2$.
\end{example} 

\begin{example}[Surfaces]\label{E:surf}
Let $R$ be the local ring at the origin of the surface $x^3=y^2z$. By Federer's criterion \cite{FW}, one can see that $R$ is not F-split when $p=3$. In fact, we have
$$
\frob*R \iso (x,y)^2R^3\oplus (x^2,y)R^3\oplus (x^2,yz)R^3
$$

When $p=5$, we can show that $(x^2,y)R$ is a direct summand of $\frob*R$ and also a direct summand of its own \fr\ $\frob*(x^2,y)R$, i.e., F-split, showing that $R$ is weakly F-split. 

For $p=7$, we look at the surface $x^2=y^4z^5$. In \cite{SchToric}, we introduced the notion of F-integral closure (=collection of fractions $f/g$   some $p^n$-th power of which lies in $R$), and showed that they often are already MCM's. In this case, its F-integral closure  $R':=\pol R{\tfrac x{yz^2}}$ is indeed an MCM, but it is also F-split: $R'$ is a direct summand of $\frob*R'$.
\end{example}

%

\providecommand{\bysame}{\leavevmode\hbox to3em{\hrulefill}\thinspace}
\providecommand{\MR}{\relax\ifhmode\unskip\space\fi MR }
\providecommand{\MRhref}[2]{%
  \href{http://www.ams.org/mathscinet-getitem?mr=#1}{#2}
}
\providecommand{\href}[2]{#2}

\end{document}